\documentclass[]{interact}

\usepackage{epstopdf}% To incorporate .eps illustrations using PDFLaTeX, etc.
\usepackage[caption=false]{subfig}% Support for small, `sub' figures and tables

\usepackage[numbers,sort&compress]{natbib}% Citation support using natbib.sty
\bibpunct[, ]{[}{]}{,}{n}{,}{,}% Citation support using natbib.sty
\renewcommand\bibfont{\fontsize{10}{12}\selectfont}% Bibliography support using natbib.sty
\makeatletter% @ becomes a letter
\def\NAT@def@citea{\def\@citea{\NAT@separator}}% Suppress spaces between citations using natbib.sty
\makeatother% @ becomes a symbol again

\theoremstyle{plain}% Theorem-like structures provided by amsthm.sty
\newtheorem{theorem}{Theorem}[section]
\newtheorem{lemma}[theorem]{Lemma}

\theoremstyle{definition}

\theoremstyle{remark}
\newtheorem{remark}{Remark}

\begin{document}

%\articletype{ARTICLE TEMPLATE}% Specify the article type or omit as appropriate

\title{Comment on Maps preserving the dimension of fixed points of products of operators}
%"[linear and Multilinear algebra (2013) 1285-1292.]"
\author{
\name{S. Elouazzani \textsuperscript{a}, M. Elhodaibi \textsuperscript{a}\thanks{CONTACT S. Elouazzani. Email: Elouazzani.soufiane@ump.ac.ma}, and S. Saber \textsuperscript{a}.}
\affil{\textsuperscript{a}Departement of Mathematics, Labo LIABM, Faculty of Sciences, 60000 Oujda, Morocco;}
}

\maketitle

\begin{abstract}
Let $\mathcal{B} (X)$ be the algebra of all bounded linear operators on an infinite-dimensional complex Banach space $X$. In this note, we show that a lemma used in the proof of the main result of [ Taghavi and Hosseinzadeh, linear and Multilinear algebra (2013) 1285-1292.] has an incorrect proof. Then, instead of such a lemma, we provide two elementary lemmas to obtain a proof of Taghavi and Hosseinzadeh's main result free of errors. 
\end{abstract}

\begin{keywords}
Preserver problem; operator algebra; fixed point
\end{keywords}
\hspace{0.4cm}\textbf{AMS Subject Classifications:} 46J10; 47B48
\section{Introduction}
In the following, let $\mathcal{B}(X)$ denote the algebra of all bounded linear operators on an infinite-dimensional complex Banach space. For any $x \in X$ and $f \in X^*$, let $x \otimes f$ be the operator of rank at most one defined on $X$ by $(x\otimes f)z = f(z)x$ for all $z \in X$. Note that the rank one operator $x\otimes f$ is idempotent if and only if $f(x) = 1$ and is nilpotent if and only if $f(x) = 0.$
Denote by $\mathcal{F}_1(X), \mathcal{P}_1(X)$, $\mathcal{N}_1(X)$ and $\mathcal{P}(X)$ the set of all at most rank one operators, the set of all rank one idempotent operators, the set of all at most rank one nilpotent operators and the set of all idempotent operators in $\mathcal{B}(X)$, respectively. For an operator $T\in \mathcal{B}(X)$, a vector $x\in X$ is a fixed point of $T$, whenever we have $Tx= x$. The set of all fixed points of T will be denoted by $F(T)$.
Given $P, Q \in \mathcal{P}_1(X)$, we say that $P$ and $Q$ are orthogonal if $PQ=QP=0$. \\
Throughout this note, assume that $\phi$ is a surjective map on $\mathcal{B}(X)$ satisfying
\begin{equation}
    \operatorname{dim}(F(\phi(T) \phi(S)))=\operatorname{dim}(F(T S)),\;(S,\; T \in \mathcal{B}(X)) .
\end{equation}
In [4], Taghavi and Hosseinzadeh characterized such a map $\phi$ and their main result states that there exists an invertible operator $A \in \mathcal{B}(X)$ such that $\phi(T)=$ $A T A^{-1}$ for all $T \in \mathcal{B}(X)$, or $\phi(T)=-A T A^{-1}$ for all $T \in \mathcal{B}(X)$. Its proof uses several lemmas which were established in the aforecited paper. Among them, we mention here Lemma 2.11 which states that such a map $\phi$ preserves the orthogonality of rank-one idempotents in both directions. In the proof of such a lemma, Taghavi and Hosseinzadeh considered two rank-one idempotents $P$ and $Q$ such that $P Q=Q P=0$ and showed that both $\phi(P)$ and $\phi(Q)$ are too rank-one idempotents for which there is a bijective function $\alpha$ on $\mathbb{C}$ such that
\begin{equation}\label{eqpq}
    \operatorname{dim}(F(\alpha(\lambda) \phi(P) \phi(Q)))=\operatorname{dim}(F(\alpha(\lambda) \phi(Q) \phi(P)))=0
\end{equation}
 Then they claimed that if $\phi(P) \phi(Q)=\phi(Q) \phi(P)=0$ don't hold, there is a nonzero $x \in X$ and nonzero $f \in X^*$ such that $f(x) \neq 0$ and
$$
\phi(P) \phi(Q)=x \otimes f \text {. }
$$ Because they claimed that $f(x) \neq 0$, there is a nonzero scalar $\lambda_0 \in \mathbb{C}$ such that $\alpha\left(\lambda_0\right)=\frac{1}{f(x)}$ and thus $1=\operatorname{dim}\left(F\left(\alpha\left(\lambda_0\right) \phi(P) \phi(Q)\right)\right)$. This contradicts (\ref{eqpq}) and shows that
$$
\phi(P) \phi(Q)=\phi(Q) \phi(P)=0 ;
$$ as they wished. Note that if $\phi(P) \phi(Q)=\phi(Q) \phi(P)=0$ don't hold, then either
\begin{enumerate}
    \item[(1)] $\phi(P) \phi(Q) \neq 0$ and $\phi(Q) \phi(P) \neq 0$, or
\item[(2)]  $\phi(P) \phi(Q) \neq 0$ and $\phi(Q) \phi(P)=0$, or
\item[(3)] $\phi(P) \phi(Q)=0$ and $\phi(Q) \phi(P) \neq 0$.
\end{enumerate}
If (1) holds, then $f(x) \neq 0$ and the above arguments are correct. If, however, (2) or (3) holds then $f(x)=0$. And so, these two similar cases were not discussed by Taghavi and Hosseinzadeh and the above arguments won't work.

\section{Preliminaries and Correct version}
First of all, let $\phi: \mathcal{B}(X) \longrightarrow$ $\mathcal{B}(X)$ be a surjective map which satisfies
\begin{equation}\label{eq3}
    \operatorname{dim} F(ST)=\operatorname{dim} F(\phi(S) \phi(T)) \quad(S, T\in \mathcal{B}(X)).
\end{equation}
 Let $x, y \in X$ and $f,g \in X^*$, we say that $x\otimes f \sim y \otimes g$ if $x$ and $y$ are linearly dependent or $f$ and $g$ are linearly dependent.\\

\begin{remark}
In line 8 of the proof of \cite[Lemma 2.11]{ref4}, the authors took that $\alpha(\lambda)= f(x)^{-1}$, which is impossible if $\phi(P)\phi(Q)=x\otimes f $ is nilpotent rank one operator. 
In that way, by the proof of \cite[Lemma 2.11]{ref4}, if $P$ and $Q$ are rank-one idempotents such that $PQ= QP= 0$, we can easily observe that $$\dim F(\phi(P)\phi(Q))= \dim F(\phi(Q)\phi(P))=0,$$ and so from this and \cite[Lemma 2.6]{ref4} we have $\phi(P) \phi(Q)$ is a rank one nilpotent operator, which means that $\phi(P)\phi(Q)=x \otimes f $ with $f(x) =0$. Hence, we cannot put that $\alpha(\lambda) = f(x)^{-1}.$ And this was exactly the error made in \cite[Lemma 2.11]{ref4}.
\end{remark}
The next Lemma quoted from \cite{ref2}. 
\begin{lemma}(\cite[Proposition 2.7]{ref2})\label{lem2.11}
 Let $A$ and $B$ be linearly independent rank one operators. Then the following are equivalent.
\begin{enumerate}
\item[(i)] $A \sim B$.
\item[(ii)] There exists a $R \in \mathcal{F}_1(X)$ such that $R$ is linearly independent of $A$, of $B$ and for every $T \in \mathcal{B}(X)$ we have $AT, BT \in \mathcal{P}(X) \backslash\{0\}$ imply $RT \in \mathcal{P}(X) \backslash\{0\}$.
\end{enumerate}
\end{lemma}
The following Lemma quoted from \cite{ref3}.
\begin{lemma}(\cite[Theorem 2.4]{ref3})\label{lem2}
Let $X$ be an infinite-dimensional Banach space and $\xi: \mathcal{P}_1(X) \rightarrow$ $\mathcal{P}_1(X)$ a bijective map preserving orthogonality in both directions. Then either there exists a bounded invertible linear or conjugate-linear operator $T: X \rightarrow X$ such that
$$
\xi(P)=T P T^{-1}, \quad P \in  \mathcal{P}_1(X),
$$
or there exists a bounded invertible linear or conjugate-linear operator $T: X^{*} \rightarrow$ $X$ such that
$$
\xi(P)=T P^{*} T^{-1}, \quad P \in  \mathcal{P}_1(X).
$$
In the second case $X$ must be reflexive.
\end{lemma} 

The next three Lemmas quoted from \cite{ref4}.
\begin{lemma}(\cite[Lemma 2.7]{ref4})\label{lem27}
     $\phi$ and $-\phi$ preserve rank one idempotent operators in both directions.
\end{lemma}
\begin{lemma}(\cite[Lemma 2.9]{ref4})\label{lem29}
    $\phi$ preserves rank one nilpotent operators in both directions.
\end{lemma}

\begin{lemma}(\cite[Lemma 2.10]{ref4})\label{lem210}
  $\phi (\lambda P) = \alpha ( \lambda) \phi(P)$ for every $P \in \mathcal{P}_1 (X)$ and all $\lambda \in \mathbb{C}$, where $\alpha $ is a bijective function on $\mathbb{C}$.
\end{lemma}

\begin{lemma}\label{rmk1}
    The assertion in Lemma \ref{lem210} is hold for every rank one nilpotent operator, too. 
    \end{lemma}

    \begin{proof}
        With no extra efforts, just as in the proof of \cite[Lemma 3.6]{ref2}, we can get that $\phi(\lambda N) = \alpha(\lambda) \phi(N)$ for every $N \in \mathcal{N}_1 (X)$.
    \end{proof}
    
The Lemma 2.11 in \cite{ref4} should be replaced by the following results.\\\\
The following Lemma characterizes when two rank one operators are equivalent.
\begin{lemma}\label{lem6}
Let $A$ and $B$ be linearly independent rank one operators. Then the following are equivalent.

\begin{enumerate}
\item[(i)] $A \sim B$.
\item[(ii)] There exists a $R \in \mathcal{F}_1(X)$ such that $R$ is linearly independent of $A$, of $B$ and for every $T \in \mathcal{B}(X)$ we have $\dim F(AT ) =\dim F(BT)= 1 $ imply $\dim F(RT)=1$.
\end{enumerate}
\end{lemma}
\begin{proof}
$(i) \Rightarrow (ii)$ By Lemma \ref{lem2.11}, There exists a $R \in \mathcal{F}_1(X)$ such that $R$ is linearly independent of $A$ and of $B$. On the other hand, for every $T \in \mathcal{B}(X)$, if $\dim F(A T)=\dim F(BT)=1$. Then, since $AT, BT \in \mathcal{F}_1(X)$, we have $AT, BT \in \mathcal{P}_1(X) \backslash\{0\} \subset \mathcal{P(X)} \backslash\{0\}$. Lemma \ref{lem2.11} implies that $RT \in \mathcal{P}(X) \backslash\{0\}$. Since $R \in \mathcal{F}_1(X), $ hence $R T \in \mathcal{P}_{1}(X) \backslash\{0\}$, and thus $\dim F(RT) = 1.$

$(ii)\Rightarrow (i)$ It is clear that, there exists a $R \in \mathcal{F}_1(X)$ such that $R$ is linearly independent of $A,$ of $B$ and for every $T \in \mathcal{B}(X)$ we have $AT, BT \in \mathcal{P}(X)\backslash \{0 \}$ imply that $RT \in \mathcal{P}(X) \backslash \{ 0 \}$. From Lemma \ref{lem2.11}, we get that $A\sim B$.
\end{proof}

%\begin{remark}\label{rem2}
   % The lemma \ref{lem7} implies that a bijective map $\phi$ which satisfies condition $P \sim Q $ if and only if $\phi(P) \sim \phi(Q),$ for every $P, Q \in \mathcal{F}_1 (X)$, also preserves orthogonality of rank-one operator in both directions.
%\end{remark}

\begin{lemma}\label{lem8}
    Let $A, B \in \mathcal{F}_1(X)$. Then $A\sim B$ if and only if $\phi (A) \sim \phi(B)$.
\end{lemma}
\begin{proof}
Let $A\sim B$. Lemma \ref{lem6} implies that there exists a $R \in \mathcal{F}_1(X)$ such that $R$ is linearly independent of $A,$ of $B$ and for every $T \in \mathcal{B}(X)$ we have $\dim F(AT)= \dim F(BT)=1 $ imply that $\dim F( RT) =1$. By Lemma \ref{lem210} and Lemma \ref{rmk1}, we get that $\phi(R)$ is linearly independent of $\phi(A)$ and of $\phi(B)$. Moreover, for every $T \in \mathcal{B}(X)$ we have $\dim F(\phi\left( A \right) \phi(T))=\dim F(\phi(B) \phi(T)) =1$ imply that $\dim F( \phi(R) \phi(T)) =1$. It follows from Lemma \ref{lem2.11} and the surjectivity of $\phi$ that $\phi\left(A\right) \sim \phi\left(B\right)$. By considering $\phi^{-1}$, the converse is similar. The proof is then complete.
\end{proof}

The next Lemma is quoted from \cite{ref1}.
\begin{lemma}(\cite[Lemma 2.12]{ref1}) \label{lem7}
Let $P$ and $Q$ with $P \neq Q$ be rank one idempotents. Then the following are equivalent.
\begin{enumerate}
    \item[(i)] $P$ and $Q$ are orthogonal.
    \item[(ii)] There exist rank one nilpotent operators $M$ and $N$ such that $P \sim N, P \sim M, Q \sim N, Q \sim M$ and $M \nsim N$.
\end{enumerate}
\end{lemma}
Automatically, the proof of the Lemma 2.12 in \cite{ref4} changes.
\begin{lemma} 
Assume that $\phi: \mathcal{B}(X) \longrightarrow$ $\mathcal{B}(X)$ is a map satisfying (\ref{eq3}). Then one of the following situations holds.
\begin{enumerate}
    \item[(i)] There exists a bounded bijective linear or conjugate linear operator $S: X \rightarrow X$ such that 
\begin{equation}
    \phi(P)=S P S^{-1}
    \label{eq1}
\end{equation} for every $P \in \mathcal{P}_1(X)$.
\item[(ii)] $X$ is reflexive and there exists a bounded bijective linear or conjugate linear operator $S: X^* \rightarrow X$ such that \begin{equation} \label{eq2}
    \phi(P)=S P^* S^{-1}
\end{equation} for every $P \in \mathcal{P}_1(X)$.
\end{enumerate}
\end{lemma}
\begin{proof}
    Applying Lemma \ref{lem27}, Lemma \ref{lem29}, Lemma \ref{lem8} and Lemma \ref{lem7}. We get that $\phi$ preserves the orthogonality of rank one idempotents in both directions.
    Moreover, Lemma \ref{lem2} implies that $\phi$ have the form (\ref{eq1}) or (\ref{eq2}).\\\\
\end{proof}

\begin{remark}
     Taghavi and Hosseinzadeh's arguments are correct when $X=H$ is a Hilbert space. Indeed, consider two rank-one idempotents $P$ and $Q$ in $\mathcal{B}(H)$ such that $P Q=Q P=0$ and as showed by Taghavi and Hosseinzadeh both $\phi(P)$ and $\phi(Q)$ are too rank-one idempotents. Note that, since $\phi(P) \phi(Q)=0$ implies $0=(\phi(P) \phi(Q))^{*}=\phi(Q)^{*} \phi(P)^{*}=\phi(Q) \phi(P)$, one observes that (2) and (3) can not occur.
\end{remark}

\end{document}